\documentclass[a4paper,10pt]{amsart}
\usepackage[utf8]{inputenc}

\usepackage{amsfonts,amsmath,amsthm,amssymb, mathrsfs}
\usepackage[numbers]{natbib}
\usepackage[american]{babel}
\usepackage{mathrsfs}
\usepackage{tikz}
\usepackage{hyperref}
\hypersetup{
   colorlinks  = true,
   urlcolor    = black,
   linkcolor   = black,
   citecolor   = black
}
\usepackage{csquotes}

\theoremstyle{plain}
\newtheorem{fact}{FACT}[section]
\newtheorem{theorem}[fact]{Theorem}

\newtheorem{lemma}[fact]{Lemma}
\newtheorem{proposition}[fact]{Proposition}
\newtheorem{corollary}[fact]{Corollary}
\theoremstyle{definition}
\newtheorem{definition}[fact]{Definition}

\newtheorem{rmks}[fact]{Remarks}
\newtheorem{example}[fact]{Example}

\newcommand{\LL}{\mathscr{L}}
\newcommand{\EE}{\mathscr{E}}
\newcommand{\GG}{\mathcal{G}}

\newcommand{\dd}{\operatorname{d}}

\newcommand{\CS}{\mathscr{S}}
\newcommand{\TG}{\mathscr{T}_G}    
\newcommand{\mJ}{\mathcal{J}^*}

\newcommand{\ZZ}{\mathbb{Z}}
\newcommand{\NN}{\mathbb{N}}
\newcommand{\KK}{\mathbb{K}}
\newcommand{\bt}{\mathbf{t}}

\newcommand{\dcup}{\operatorname{\mathbin{\dot{\cup}}}}

\title[On the minimal generating sets of the Eulerian ideal]{On the minimal generating sets \\ of the Eulerian ideal}
\author{Jorge Neves}
\address[Jorge Neves]{University of Coimbra, Department of Mathematics, CMUC, 3000-143 Coimbra, Portugal}
\email{neves@mat.uc.pt}

\author{Gon\c{c}alo Varej\~{a}o}
\address[Gon\c{c}alo Varej\~{a}o]{University of Coimbra, Department of Mathematics, CMUC, 3000-143 Coimbra, Portugal}
\email{g.varejao@mat.uc.pt}

\thanks{Partially supported by the Centre for Mathematics of the University of Coimbra (funded by the Portuguese Government through FCT/MCTES, DOI 10.54499/UIDB/00324/2020).
Partially supported by the Centre for Mathematics of the University of Coimbra - UIDB/00324/2020,
funded by the Portuguese Government through FCT/MCTES and by the Portuguese Republic / MCTES,
the ESF, and the POR Centro, through FCT – Funda\c{c}\~{a}o para a Ci\^{e}ncia e a Tecnologia, I.P., DOI 10.54499/2021.05420.BD.
The authors use Macaulay2 \cite{M2} in the computations of examples.}

\thanks{\emph{Data availability}. Data sharing not applicable to this article as no datasets were generated or analyzed during the current study.}

\begin{document}

\maketitle

\begin{abstract}
We study the minimal homogeneous generating sets of the Eulerian ideal associated with a simple graph 
and its maximal generating degree. We show that the Eulerian ideal is a lattice ideal and use this to  
give a characterization of binomials belonging to a minimal homogeneous generating set. In this way, 
we obtain an explicit minimal homogeneous generating set. We find an upper bound for the 
maximal generating degree in terms of the graph. This invariant is half the number of 
edges of a largest Eulerian subgraph of even cardinality without even-chords. 
We show that for bipartite graphs this invariant is the maximal generating degree. 
In particular, we prove that if the graph is bipartite, the Eulerian ideal 
is generated in degree $2$ if and only if the graph is chordal. Furthermore, 
we show that the maximal generating degree is also $2$ when the graph is a complete 
graph.
\end{abstract}

\section{Introduction}

Currently, one of the most popular research areas in commutative algebra is the
study of properties of ideals, of the ring of polynomials over a field, associated with graphs. 
This originated with the famous paper \cite{OtIToG} and has 
motivated the study of several other ideals associated with graphs. 
Among these, the Eulerian ideal is a recent addition and the focus of this paper.
It was defined in \cite{NePiVi}, by the first author, together with Vaz Pinto and Villarreal, 
inspired by the vanishing ideal parameterized by a graph, 
introduced in \cite{ReSiVi}. The Eulerian ideal is a binomial ideal of a polynomial 
ring whose variables are indexed by the edges of a simple graph. 
It is known that this ideal has a generating set that contains a set 
of binomials canonically associated with the Eulerian subgraphs of the graph --- 
see Subsection~\ref{subsec: the ideal} --- and that, moreover, this ideal contains 
the toric ideal associated with the graph.

Several properties of Eulerian ideals have been studied.
In \cite{NePiVi}, the quotient of the polynomial ring by it was
shown to be a Cohen-Macaulay graded ring of dimension $1$; in \cite{Neves}, a Gröbner basis for the ideal was given and
a combinatorial characterization of its Castelnuovo-Mumford regularity and degree was derived; in \cite{NeVa}, the ideal and the above results were generalized for $k$-uniform hypergraphs 
and the Hilbert function of the quotient was characterized by 
a combinatorial formula and computed in some cases; finally, in \cite{Neves1},
a monomial basis of the socle of the quotient and a description of the socle degrees was found.  

Through all of these results the graph invariant that appears in common, and that translates the properties of the Eulerian ideal, is the notion of $T$-join. Given a subset, $T$, of vertices of the graph, 
a $T$-join is a subset of edges, $J$, such that, in the subgraph with edge set $J$, $T$ is the set of odd degree vertices. We will focus on $T$-joins with cardinality of a fixed parity. For more on $T$-joins, see \cite{KoVyCoOp} and Subsection~\ref{subsec: joins}, below.

This work has two main goals; to characterize minimal binomial generating sets of the Eulerian ideal and to study its  
maximal generating degree. In the case of the toric ideal of a graph, an explicit description of the minimal generating sets, in terms of the closed even walks of the graph, was obtained in \cite[Theorem 4.13]{ReTaTh}. For the Eulerian ideal,
we have found that a minimal generating set consists of binomials that identify $T$-joins that are not equivalent under a certain equivalence relation --- see
Definition~\ref{def: equivalent m c (T,p)-joins} and
Theorem~\ref{thrm: necessary condition for minimal binomials}.  
The equivalence relation is an adaptation of the notion of Markov basis, as introduced by Diaconis and Sturmfels in \cite{DiSt} and used in \cite[Theorem 4.12]{CTV1} to obtain a description of all minimal generating sets of binomials
of a lattice ideal. Since, as we show in Proposition~\ref{prop: lattice description}, 
the Eulerian ideal is a lattice ideal, in Theorem~\ref{thrm: the minimal generating set of IG},
we use this result to obtain an explicit minimal generating set for the Eulerian ideal.

Regarding the maximal generating degree, we start by showing that this invariant 
is $2$ for complete graphs --- see Theorem~\ref{thrm: complete graphs}. 
Then, in Theorem~\ref{thrm: d(IG)=d}, by showing that the binomials associated with
even cycles with an even-chord or with edge-disjoint unions of two odd cycles with at most 
one vertex in common that possess an even-chord are not needed in a minimal generating set of the Eulerian ideal, we 
obtain an upper bound for the maximal generating degree.
This result is analogous to the results \cite[Lemmas 3.1, 3.2 and Theorem 1.2]{OH1999} for the toric ideal.
As an application, in Corollary~\ref{cor: chordal bipartite graphs}, we show that if the graph is bipartite, then 
the Eulerian ideal is generated in degree $2$ if and only if the graph is chordal.

This paper is organized as follows.
In Section~\ref{sec: Preliminaries}, we define the ideal, the graph theoretical notions that we need, exhibit the lattice associated with the Eulerian ideal and recall two background results: the characterization of its binomials and the description of a Gröbner basis.
In Section~\ref{sec: A minimal generating set}, we characterize when a binomial in the 
Eulerian ideal is in the ideal generated by binomials of lower degree 
and we describe a minimal generating set.
In Section~\ref{sec: Generating degrees}, we give an upper bound for the maximal generating degree of the Eulerian ideal and show that this invariant is $2$ for bipartite chordal graphs and complete graphs.

\section{Preliminaries}\label{sec: Preliminaries}

Let $G=(V_G,E_G)$ be a simple graph. Let us assume that $V_G=\{1,\dots,n\}$ and that 
$s=|E_G|>0$. Let $\KK$ denote a field and let us  
consider the polynomial rings $\KK[x_1,\dots,x_n]$ and 
$\KK[E_G]=\KK[t_e : e\in E_G]$. Throughout, we use the multi-index notation for monomials in a polynomial 
ring, so that, for example, if $\alpha\in \NN^{E_G}$, 
$$\textstyle\bt^\alpha=\prod\limits_{e\in E_G} t_e^{\alpha(e)}.$$
Also, if $J\subseteq E_G$, we denote by $\bt_J$ the product of all 
variables $t_e$, with $e\in J$.

\subsection{The ideal}\label{subsec: the ideal}
Let $\varphi$ be the ring homomorphism $\KK[E_G]\to \KK[x_1,\dots,x_n]$ that maps $t_e \mapsto x_i x_j$, for every $e=\{i,j\}$ in $E_G$. 
\begin{definition}\label{def: the ideal}
The Eulerian ideal of $G$ is the ideal of $\KK[E_G]$ defined by
$$I(G)=\varphi^{-1}(x_i^2-x_j^2 : i,j\in V_G).$$
\end{definition}
Given two edges of $G$, $e=\{i,j\}$ and $\ell=\{u,v\}$, it is easy to see that the binomial 
$t_{e}^2-t_{\ell}^2$ is in $I(G)$,
$$\textstyle\varphi(t_e^2-t_\ell^2)=x_i^2 x_j^2-x_u^2x_v^2=x_i^2 (x_j^2-x_u^2)+x_u^2 (x_i^2-x_v^2).$$
Therefore the Eulerian ideal is trivial if and only $G$ has only one edge.
From now on let us assume that $|E_G|\geq 2$. 

To describe a generating set of $I(G)$, let us fix some terminology.
For a set of edges $J\subseteq E_G$ and a vertex $i\in V_G$, we define the degree 
of $i$ in $J$ as 
$$
\textstyle\deg_J (i)= \sum\limits_{e\in J} |e \cap \{i\}|.
$$
We say that a subset of edges $C\subseteq E_G$ is Eulerian if $\deg_C (i)$ is even, for every $i\in V_G$. 
Given an Eulerian subset $C$ of even cardinality and a partition $C=J\dcup K$ into 
sets of same cardinality, we form a homogeneous binomial in $\KK[E_G]$,
$$
\textstyle \bt_J-\bt_K.
$$
Binomials obtained in this way are called Eulerian binomials associated with $C$.
Let us denote by $\EE$ the set of all Eulerian binomials, associated with 
all such $C$ and choices of partitions. We recall from \cite[Theorem 3.3]{Neves} the following result.
\begin{theorem}
The set 
$
\{t_e^2-t_\ell^2 : e,\ell\in E_G\}\cup \EE
$
is a Gröbner basis for $I(G)$, with respect to 
the graded reverse lexicographic order on $\KK[E_G]$, associated with any choice 
of ordering of $E_G$.
\end{theorem}

\subsection{Joins}\label{subsec: joins} 
Given $T\subseteq V_G$, a $T$-join is a subset of edges, $J\subseteq E_G$, such that  
$$
T=\{v\in V_G : \deg_J (v) \text{ is odd}\}.
$$
If $p\in \ZZ_2$, we say that $J$ is a $(T,p)$-join whenever $|J|+2\ZZ=p$.
Throughout, we abbreviate $0+2\ZZ$ and $1+2\ZZ$ by $0$ and $1$, respectively. 
It is clear that, given any subset of edges $J\subseteq E_G$, there is a unique pair
$(T,p)$ such that $J$ is a $(T,p)$-join. For example, an Eulerian subset of even cardinality 
is an $(\emptyset,0)$-join. However, given a pair $(T,p)$, there may or may not exist a $(T,p)$-join.
For one, by \cite[Proposition 12.7]{KoVyCoOp}, if a $(T,p)$-join exists,
then the intersection of $T$ with the vertex set of each connected component of the graph must have even cardinality.

\begin{example}
Let us consider the graph in Figure~\ref{fig: A bipartite graph}.
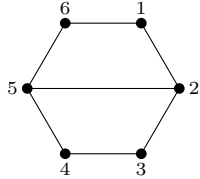
\begin{figure}[ht]
\begin{center}
\begin{tikzpicture}

\coordinate (P1) at (.5,0.8660254038);
\coordinate (P2) at (1,0);
\coordinate (P3) at (.5,-0.8660254038);
\coordinate (P4) at (-.5,-0.8660254038);
\coordinate (P5) at (-1,0);
\coordinate (P6) at (-.5,0.8660254038);

\draw (P1)--(P2);
\draw (P1)--(P6);
\draw (P2)--(P3);
\draw (P3)--(P4);
\draw (P4)--(P5);
\draw (P5)--(P6);
\draw (P2)--(P5);

\foreach \i in {1,...,6} {\fill[color=black] (P\i) circle (2pt);}

\foreach \i in {1} {\draw  (P\i) node[anchor=south] {$\scriptstyle \i$};}
\foreach \i in {2} {\draw  (P\i) node[anchor=west] {$\scriptstyle \i$};}
\foreach \i in {3} {\draw  (P\i) node[anchor=north] {$\scriptstyle \i$};}
\foreach \i in {4} {\draw  (P\i) node[anchor=north] {$\scriptstyle \i$};}
\foreach \i in {5} {\draw  (P\i) node[anchor=east] {$\scriptstyle \i$};}
\foreach \i in {6} {\draw  (P\i) node[anchor=south] {$\scriptstyle \i$};}

\end{tikzpicture}
\end{center}
\caption{A hexagon with an even-chord.} 
\label{fig: A bipartite graph}
\end{figure} 
Set $T=\{2,5\}$.
Then $\{\{2,5\}\}$ is a $(T,1)$-join, as are 
$\{\{1,2\},\{1,6\},\{5,6\}\}$ and $\{\{2,3\},\{3,4\},\{4,5\}\}.$
However, this graph has no $(T,0)$-joins, for $T=\{2,5\}$.
\end{example}

Throughout, let us denote by $\Delta$ the symmetric difference of two sets.
By an elementary argument (see, for example, \cite[Proposition 12.6]{KoVyCoOp}), 
one can show that if $J$ is a $T$-join and $J'$ is a $T'$-join
then $J\Delta J'$ is a $(T\Delta T')$-join. It follows that if $J$ is a $(T,p)$-join
and $J'$ is a $(T',p')$-join then $J\Delta J'$ is a $(T\Delta T',p+p')$-join.
It is then easy to show that, when $G$ is bipartite,
either there exists a $(T,0)$-join or there exists a $(T,1)$-join and, 
when $G$ is non-bipartite, there exist both a $(T,0)$-join and a $(T,1)$-join
(see \cite[Lemma 4.6]{Neves}). 

\smallskip

The relation between $(T,p)$-joins and $I(G)$ may be expressed as follows.

\begin{proposition}\label{prop: on the fibers}
For any monomials $\bt^\alpha=\bt^{2\mu}\bt_J$ and $\bt^\beta=\bt^{2\nu}\bt_K$, of $\KK[E_G]$, 
let 
$(T_1,p_1)$ and $(T_2,p_2)$ be such that 
$J$ is a 
$(T_1,p_1)$-join and $K$ is a  $(T_2,p_2)$-join.
Then $\bt^{\alpha}-\bt^{\beta}\in I(G)$
if and only if $(T_1,p_1)=(T_2,p_2)$ and $\bt^{\alpha}-\bt^{\beta}$ is homogeneous.
\end{proposition}
\begin{proof}
See the proof of \cite[Proposition 3]{NeVa}. 
\end{proof}

By \cite[Theorem 4.3]{Neves}, the reduction of a monomial by the Gröbner basis 
identifies a unique $(T,p)$-join of minimal cardinality. Moreover, 
$(T,p)$-joins of minimum cardinality are parity joins in the sense of 
\cite[Definition 4.11]{Neves}.
These notions are used to express the regularity of $I(G)$ --- 
\cite[Theorem 4.13]{Neves} --- and also to identify 
a set of standard monomials for the quotient of the polynomial ring by $I(G)$
--- \cite[Theorem 4.9]{Neves} and \cite[Proposition 2.7]{Neves1}.

\begin{definition}
Let us denote by $\TG$ the set of all pairs $(T,p)$ for which there exists a 
$(T,p)$-join.
Given a pair $(T,p)\in \TG$, we denote by
$\mJ(T,p)$ the set of all $(T,p)$-joins with minimum cardinality.  
\end{definition}

\begin{proposition}\label{prop: char parity joins}
Given $(T,p)\in \mathscr{T}_G$ and $J$ a $(T,p)$-join, the following are equi\-va\-lent:
\begin{enumerate}
\item $J\in \mJ(T,p)$;
\item for every even cardinality Eulerian subset $C\subseteq E_G$,
$$\textstyle|J\Delta C|\geq |J| \iff |J\cap C|\leq \frac{|C|}{2}.$$
\end{enumerate}
\end{proposition}
\begin{proof}
If $J$ is a $(T,p)$-join and $C$ is an even cardinality Eulerian subset then 
$J\Delta C$ is also a $(T,p)$-join.
Hence $(i) \Rightarrow (ii)$.
If $K$ is another $(T,p)$-join then $J\Delta K$ is an even cardinality Eulerian set and
hence $(ii)\Rightarrow (i)$.
\end{proof}

\subsection{The lattice property}\label{subsec: The lattice property}
Fix $e_1,\dots,e_s$ an ordering of the edge set $E_G$.
The incidence matrix 
of $G$ is the $n\times s$ matrix $B$, where each entry $B_{ij}$ equals $1$ if the vertex $i$ is in the edge $e_j$, and $0$ otherwise. 
To identify the lattice associated with $I(G)$, let us use 
the induced identification of $\NN^{E_G}$ with $\NN^s$.

\begin{proposition}\label{prop: lattice description}
$I(G)$ is the lattice ideal associated with the lattice 
$$\textstyle \LL=\{\theta \in \ZZ^s : B\theta\in (2\ZZ)^n \text{ and } \sum_{e\in E_G} \theta(e)=0 \}.$$ 
\end{proposition}
\begin{proof}
Recall that, given $\LL\subseteq \ZZ^s$ a lattice, 
the lattice ideal associated with $\LL$, which we denote by 
$I_\LL$, is the ideal generated by 
$\bt^\alpha-\bt^\beta$, where $\alpha,\beta\in \NN^s$ and $\alpha-\beta\in \LL$.
Given $\alpha,\beta \in \NN^{E_G}\cong \NN^s$, let 
$\mu, \nu \in \NN^{E_G}$ and $J, K \subseteq E_G$ be such that  
$\bt^\alpha=\bt^{2\mu}\bt_J$ and $\bt^\beta=\bt^{2\nu}\bt_K$.
By Proposition~\ref{prop: on the fibers}, the Eulerian ideal is generated by the 
set of binomials $\bt^\alpha-\bt^\beta$ such that 
\begin{equation}\label{eq: homogeneous condition}
\textstyle \deg(\bt^\alpha)=\deg(\bt^\beta) \iff \sum\limits_{e\in E_G}\bigl( \alpha(e)-\beta(e)\bigr)=0
\end{equation}
and such that $J$ and $K$ are $(T,p)$-joins, for some $(T,p)\in \TG$. 
Let us show that this set coincides with the generating set of $I_\LL$. 
We have
$$
\textstyle (B\alpha)_i= \sum\limits_{j=1}^s B_{ij} \cdot \alpha(e_j)=  \sum\limits_{j : \ i\in e_j} \alpha(e_j)=\sum\limits_{j : \ i\in e_j} |\{e_j\}\cap J|+2\mu(e_j).
$$
Since, for each $i\in V_G$, the sum $\sum_{j : \ i\in e_j} | \{e_j\}\cap J|$ counts 
the number of edges in $J$ that contain $i$, which is equal to $\deg_J (i)$, we get 
$(B\alpha)_i \equiv_2 \deg_J (i)$. Hence
\begin{equation}\label{eq: incidence matrix product}
\textstyle(B(\alpha-\beta))_i=(B\alpha)_i-(B\beta)_i\equiv_2 \deg_J (i)-\deg_K (i).
\end{equation} 
Now, if $\bt^\alpha = \bt^{2\mu}\bt_J$ and $\bt^\beta=\bt^{2\nu}\bt_K$ are such that 
\eqref{eq: homogeneous condition} is satisfied and $J$ and $K$ are $(T,p)$-joins then,
by \eqref{eq: incidence matrix product}, $\bigl (B(\alpha-\beta)\bigr )_i$ is even, 
for every $i\in V_G$, and hence $\alpha -\beta \in \LL$. Conversely, if $\alpha-\beta \in \LL$
then, if $J$ and $K$ are obtained as above, by \eqref{eq: homogeneous condition} they 
have cardinalities of the same parity and moreover, using \eqref{eq: incidence matrix product}, we deduce that  
there exists $(T,p)\in \TG$ such that both $J$ and $K$ are $(T,p)$-joins.
\end{proof}

\section{A minimal generating set}\label{sec: A minimal generating set}

Consider $G$ as in Figure~\ref{fig: A bipartite graph}. 
This graph contains three Eulerian subsets of edges, two of which are squares and one 
of which is a hexagon.
Accordingly, $I(G)$ is generated by 
$\{t_e^2-t_\ell^2 : e,\ell\in E_G\}$ together with the set 
of Eulerian binomials associated with these Eulerian subsets.
However, the generating set of binomials obtained is not minimal.
Using Macaulay2, one can check that the degree $3$ Eulerian binomials 
associated with the hexagon,
$$
\renewcommand{\arraystretch}{1.3}
\begin{array}{cc}
t_{23}t_{34}t_{45}-t_{12}t_{16}t_{56}, &
t_{16}t_{34}t_{45}-t_{12}t_{23}t_{56}, \\
t_{12}t_{34}t_{45}-t_{16}t_{23}t_{56}, &
t_{16}t_{23}t_{45}-t_{12}t_{34}t_{56}, \\
t_{12}t_{23}t_{45}-t_{16}t_{34}t_{56}, &
t_{12}t_{16}t_{45}-t_{23}t_{34}t_{56}, \\
t_{16}t_{23}t_{34}-t_{12}t_{45}t_{56}, &
t_{12}t_{23}t_{34}-t_{16}t_{45}t_{56}, \\
t_{12}t_{16}t_{34}-t_{23}t_{45}t_{56}, &
t_{12}t_{16}t_{23}-t_{34}t_{45}t_{56},
\end{array}
$$ 
are not needed.

\begin{definition}[{\cite[page 367]{DiSt}}]\label{def: Markov basis}
Let $\LL\subseteq \ZZ^s$ be a lattice. A Markov basis of $\LL$ is a finite set $M\subseteq \LL$ 
such that, for every $\alpha,\beta \in \NN^s$ such that $\alpha-\beta\in \LL$, there exists a finite sequence, 
$\gamma_0,\dots,\gamma_r$, of elements in $\NN^s$, where
$\gamma_0= \alpha$, $\gamma_r= \beta$ and either $\gamma_i-\gamma_{i+1}\in M$ or
$\gamma_{i+1}-\gamma_{i}\in M$, for $i=0,\dots,r-1$.
\end{definition}

A Markov basis of $\LL$ yields a generating set of binomials of the associated lattice ideal and vice-versa, see \cite[Lemma A.1]{HM}. Definition~\ref{def: Markov basis} motivates an equivalence relation on $\mJ(T,p)$.
From now on, to help in the translation, let us fix an ordering of $E_G$, $e_1,\dots,e_s$,
and work with the identification $\NN^{E_G}\cong\NN^s$.
Suppose that we have two minimum cardinality $(T,p)$-joins, $J$ and $K$, for a fixed $(T,p)\in \TG$. 
Denote $\bt^\alpha=\bt_J, \bt^\beta=\bt_K$ and let
$M$ be a Markov basis for the lattice $\LL$ associated with $I(G)$.
Then there is a sequence of monomials, 
$\bt^{\gamma_0},\dots, \bt^{\gamma_r}$, in $\KK[E_G]$ such that
$\gamma_0=\alpha$, $\gamma_r=\beta$ and 
$\gamma_i-\gamma_{i+1}\in \LL$, for all \mbox{$i=0,\dots,r-1$.}
It follows that, for all $i$, 
\begin{equation}\label{eq: induction}
\bt^{\gamma_i}-\bt^{\gamma_{i+1}} \in I(G). 
\end{equation}
When $i=0$,  we get that
$\bt_{J}-\bt^{\gamma_{1}} \in I(G)$. By Proposition~\ref{prop: on the fibers},
since $J$ has minimum cardinality,
$\bt^{\gamma_{1}}$ is then a squarefree monomial $\bt_{L^1}$, for some minimum cardinality 
$(T,p)$-join $L^1.$
By induction on $i$, using
Proposition~\ref{prop: on the fibers} and \eqref{eq: induction}, 
we find a sequence
$L^0,L^1,\dots,L^r \in \mJ(T,p)$, where
$L^0=J$, $L^r=K$ and $\bt^{\gamma_i}=\bt_{L^i}$, for $i=0,\dots,r$.

\begin{definition}\label{def: equivalent m c (T,p)-joins}
Let $J, K\in \mJ (T,p)$ be minimum cardinality $(T,p)$-joins.
If 
there exist $L^0,L^1,\dots,L^r\in \mJ (T,p)$ such that $J=L^0$, $K=L^r$ and
$L^i \cap L^{i+1}\neq\emptyset,$  for all $0\leq i\leq r-1$, then we say that $J$ is equivalent to $K$ and write 
$J \sim K$. 
\end{definition}

Clearly, for each $(T,p)\in \TG$, 
the binary relation on $\mJ (T,p)$, defined above, is an equivalence relation.
The next result characterizes the binomials of $I(G)$ that belong to some 
minimal homogeneous generating set.

\begin{theorem}\label{thrm: necessary condition for minimal binomials}
A degree $d$ homogeneous binomial in $I(G)$ is not in the ideal 
ge\-ne\-ra\-ted by the elements of $I(G)$ of degree less than $d$ if and only if it is either
of the form $t_e^2-t_\ell^2$, for some $e,\ell\in E_G$, or of the form $\bt_J-\bt_K$, where 
$J$ and $K$ are nonequivalent minimum cardinality $(T,p)$-joins, for a given $(T,p)\in \TG.$
\end{theorem}
\begin{proof}
To prove the implication \enquote{$\Rightarrow$}, take a nonzero binomial 
$$f=\bt^{2\mu}\bt_J-\bt^{2\nu}\bt_K\in I(G)$$ 
of degree $d$ and consider $(T,p)\in \TG$ such that $J$ and $K$ are $(T,p)$-joins.
Let us suppose that $f$ is not of the forms given above and prove that then
$f$ is a $\KK[E_G]$-linear combination of binomials, in $I(G)$, having degree less than $d$.
Without loss of generality, we may assume that the monomials of $f$ are coprime. 
Then $J\cup K$ is a disjoint union and is an even cardinality Eulerian set.
If $d=2$, either $J\dcup K=\emptyset$ and $f=t_e^2-t_\ell^2$, for some $e,\ell\in E_G$, or $J\dcup K$ forms a square. In the latter, $\mu=0=\nu$ and both $J$ and $K$ satisfy Proposition~\ref{prop: char parity joins}, thus are 
minimum cardinality $(T,p)$-joins. Furthermore $J$ and $K$ are nonequivalent,
because otherwise we could find another minimum cardinality $(T,p)$-join, $N$, intersecting 
$J$ and yielding an even cardinality Eulerian set, $J\Delta N$, of cardinality $2$, which is impossible.
We suppose then that $d\geq 3$ and
let $|J|\leq |K|$ and $\deg(\bt^{2\mu})\geq  \deg(\bt^{2\nu})$.
If $K=\emptyset,$ $f=\bt^{2\mu}-\bt^{2\nu}$ is in the ideal generated by the differences of squares of 
variables. 
If $K\neq \emptyset$ and $\mu\neq 0$, taking $\ell \in K$ and an edge $e\in E_G$ such that 
$t_e^2$ divides $\bt^{2\mu}$, $f$ can be written as a combination of binomials, in $I(G)$, of 
degree less than $d$,
$$
\textstyle
\bt^{2\mu}\bt_J-\bt^{2\nu}\bt_K= \frac{\bt^{2\mu}}{t_e^2}\bt_J  \bigl(t_e^2-t_\ell^2\bigr) +
t_\ell \bigl (\frac{\bt^{2\mu}}{t_e^2}\bt_{J\dcup \ell}-\bt^{2\nu}\bt_{K\setminus \ell}\bigr).
$$
We are left with considering $K\neq \emptyset, \mu=0$ and $f=\bt_J-\bt_K$, yielding two last cases.

\smallskip

\noindent
If $|J|=|K|$ is not the minimum cardinality of a $(T,p)$-join.
By Proposition~\ref{prop: char parity joins}, 
there is an even cardinality Eulerian set 
$D\subseteq E_G$ such that 
\begin{equation}\label{eq: not a parity join}
\textstyle|J\cap D|> \frac{|D|}{2}. 
\end{equation}
We can then choose $A\subseteq J\cap D$ such that 
$|A|=\frac{|D|}{2}$.
It follows that
\begin{equation}\label{eq: linear combination}
f= \bt_{J\setminus A}(\bt_A-\bt_{D\setminus A})+\bt_{J\setminus A}\bt_{D\setminus A}-\bt_K. 
\end{equation}
Let us write $\bt_{J\setminus A}\bt_{D\setminus A}=\bt^{2\xi}\bt_L$, for some $\xi\in \NN^{E_G}$ and $L\subseteq E_G$.
Since $f$ and $\bt_A-\bt_{D\setminus A}$ belong to $I(G)$,
\eqref{eq: linear combination} implies that the binomial 
$\bt^{2\xi}\bt_L-\bt_K$ is in $I(G)$.
By \eqref{eq: not a parity join}, we deduce that $\xi\neq 0$ and so,
by the previous argument, $\bt^{2\xi}\bt_L-\bt_K$ is in the ideal 
generated by the binomials, in $I(G)$, with degrees less than $d$.

\smallskip

\noindent
In the last case, $J$ and $K$ are equivalent minimum cardinality $(T,p)$-joins.
So there exist $L^0,L^1,\dots,L^r\in \mJ (T,p)$ such that $J=L^0$, $K=L^r$ and
$L^i \cap L^{i+1}\neq\emptyset,$  for all $0\leq i\leq r-1$.
Hence   
$$
\textstyle f=\bt_J-\bt_K=\sum\limits_{i=0}^{r-1} \bt_{L^i \cap L^{i+1}}(\bt_{L^i \setminus L^{i+1}}-\bt_{L^{i+1}\setminus L^{i}}).
$$

\smallskip

\noindent
Let us prove the implication "$\Leftarrow$".
Take a binomial $f$ of the form $t_e^2 - t_\ell^2$, where $e, \ell \in E_G$, or of the form $\bt_J - \bt_K$, where $J$ and $K$ are nonequivalent minimum cardinality $(T, p)$-joins for a given $(T, p) \in \TG$. We must prove that $f$ is not a $\KK[E_G]$-linear combination of 
binomials, in $I(G)$, of lesser degree.
If $\deg(f) = 2$, this is clear. So, assume $\deg(f) \geq 3$ and $f = \bt_J - \bt_K$ with $J$ and $K$ as described. We proceed by contradiction, assuming that $f$ is
in the ideal generated in degree lesser than 
$\deg(f)$.
Then we can write
\begin{equation}\label{eq: f generates in lesser degree1}
\textstyle f=\sum\limits_{i=1}^{r} c_i(\bt^{\alpha_i}-\bt^{\beta_{i}}), 
\end{equation}
where, for all $i=1,\dots,r$, $\bt^{\alpha_i}-\bt^{\beta_{i}}\in I(G)$,
$\gcd(\bt^{\alpha_i},\bt^{\beta_{i}})\neq 1$ and $c_i \in \KK$.
The monomials $\bt_J$ and $\bt_K$ appear in \eqref{eq: f generates in lesser degree1}, respectively, as 
linear combinations $\sum_i c_i\bt_{J}$ and $\sum_i c_i\bt_{K}$.
This means that the remaining monomials on the right side of \eqref{eq: f generates in lesser degree1} cancel out. 
In particular, all the monomials that are not of the form $\bt_A$, for some $A\in \mJ(T,p)$, cancel out and we may exclude them from \eqref{eq: f generates in lesser degree1}.  
However, by Proposition~\ref{prop: on the fibers}, the monomial 
$\bt^{\alpha_i}$ is not of this form if and only if $\bt^{\beta_i}$ is neither.
So, after excluding these monomials, we still obtain $f$ as
a linear combination of binomials.
Therefore, 
without loss of generality, we assume that, for all $i=1,\dots,r,$ 
there are $J^i,K^i\in \mJ(T,p)$ such that 
$\bt^{\alpha_i}=\bt_{J^i}$ and $\bt^{\beta_{i}}=\bt_{K^i}$.
Then \eqref{eq: f generates in lesser degree1} becomes 
\begin{equation}\label{eq: f generates in lesser degree2}
\textstyle f=\sum\limits_{i=1}^{r} c_i(\bt_{J^i}-\bt_{K^i}).
\end{equation}
Finally, consider the ring homomorphism 
$\psi\colon \KK[E_G]\to \KK$ that maps 
$t_e \mapsto 1$, for every $e$ belonging to some $N\in \mJ(T,p)$ equivalent to $J$, 
and $t_e \mapsto 0$, for every other edge $e\in E_G$.
Since, for $i=0,\dots,r$, $J^i \sim K^i$,
either $J^i \sim K^i \sim J$ or $J^i \nsim J$ and $K^i \nsim J$.
In the first case,
$\psi(\bt_{J^i}-\bt_{K^i})=1-1=0$.
In the second, 
$\psi(\bt_{J^i}-\bt_{K^i})=0-0=0$.
This means that the right side of \eqref{eq: f generates in lesser degree2}
is sent to zero by $\psi$.
On the other hand, since we are assuming that 
$J\nsim K$, 
$$\psi(f)=\psi(\bt_J-\bt_K)=1-0=1,$$ 
which is a contradiction.
\end{proof}

\begin{example}
Let $G$ be the graph of Figure~\ref{fig: A bipartite graph} and let us show that 
the Eulerian binomials associated with the hexagon are not needed in a generating 
set. Consider $\bt_J-\bt_K$ one such Eulerian binomial.
When $J$ is contained in one of the squares, we see 
that $(T,p)=(\{2,5\},1)$ and a minimum 
cardinality $(T,p)$-join is thus the edge $\{2,5\}$.
Hence $J$ is not a minimum cardinality 
$(\{2,5\},1)$-join and, by Theorem~\ref{thrm: necessary condition for minimal binomials}, $\bt_J-\bt_K$ is a linear combination of binomials of lesser degree.
The example is:
$$
\renewcommand{\arraystretch}{1.3}
\begin{array}{l}
t_{12}t_{16}t_{56}-t_{23}t_{34}t_{45}
\\
= t_{12}(t_{16}t_{56}-t_{12}t_{25})+t_{25}(t_{12}^2-t_{23}^2)+
t_{23}(t_{23}t_{25}-t_{34}t_{45}).
\end{array}
$$
We may therefore assume that $J$ intersects a square of $G$, say $Q\subseteq E_G$,
in exactly $2$ edges.
Then, as both squares of $G$ contain $3$ edges of the hexagon, a third edge of $Q$ 
belongs to $K$ (and not to $J$).
Let us show that $J\sim K$, which 
by Theorem~\ref{thrm: necessary condition for minimal binomials}, proves our claim.
Consider 
$J\Delta Q\subseteq E_G$. Then $$|J\Delta Q|=|J|+|Q|-2|J\cap Q|=3+4-4=3$$ and thus 
$J\Delta Q$ is another minimum cardinality $(T,p)$-join.
Clearly $J\cap (J\Delta Q) \neq \emptyset$ and $K\cap (J\Delta Q) \neq \emptyset$, 
so $J\sim K$.
An example of this is 
$$
\renewcommand{\arraystretch}{1.3}
\begin{array}{l}
t_{12}t_{34}t_{56}-t_{23}t_{45}t_{16}
\\
= t_{34}(t_{12}t_{56}-t_{16}t_{25})+
t_{16}(t_{34}t_{25}-t_{23}t_{45}).
\end{array}
$$
\end{example}

We now give an explicit description of a minimal binomial generating set of $I(G)$.
In the remainder of this article, let us denote by $\TG^*$ the set of all $(T,p)\in \TG$ for which there are at least two 
nonequivalent minimum cardinality $(T,p)$-joins.

\begin{theorem}\label{thrm: the minimal generating set of IG}
Let $G$ be a graph. Choose $\ell\in E_G$. For each $(T,p)\in \TG^*$, choose $\CS(T,p)\subseteq \mJ (T,p)$
a set of representatives
of the equivalence  classes of $\sim$ in $\mJ (T,p)$ and choose $K(T,p)\in \CS(T,p)$.
The set 
\begin{equation}\label{eq: minimal generating set}
\textstyle \{t_e^2-t_\ell^2 : e\in E_G \setminus\ell\} \cup \bigcup\limits_{(T,p)\in \TG^*} \{\bt_J-\bt_{K(T,p)} : J\in \CS(T,p)\setminus K(T,p) \}
\end{equation}
is a minimal homogeneous generating set of $I(G)$. 
\end{theorem}
\begin{proof}
Let $S$ denote the set \eqref{eq: minimal generating set}.
We will show that $S$ satisfies the conditions of \cite[Theorem~4.12]{CTV1}, 
which characterizes every minimal binomial generating set of a lattice ideal. 
Given $\alpha\in \NN^s$, consider 
$$
F_\alpha=\{\bt^\beta : \bt^\alpha-\bt^\beta\in I(G)\},
$$
the fiber of $\bt^\alpha$ associated with $I(G)$.
We define an equivalence relation on this set as follows.
Given $\bt^\mu,\bt^\nu\in F_\alpha$, set $\bt^\mu\equiv\bt^\nu$ if and only if 
$\bt^\mu-\bt^\nu$ belongs to the ideal generated by the elements of $I(G)$
of degree less than $\deg(\bt^\alpha)$. Note that, since $I(G)$ is homogeneous, 
$$\deg(\bt^\alpha)=\deg(\bt^\mu)=\deg(\bt^\nu).$$
Define $\GG(F_\alpha)$ to be the graph the vertex set of which is the set of equivalence classes of $\equiv$
on $F_\alpha$ and for which there is an edge between $[\bt^\mu]$ and  $[\bt^\nu]$ if and only if 
$\bt^\mu \not\equiv \bt^\nu$ and there is a binomial $\bt^\delta-\bt^\gamma\in S$ such that 
$\bt^\delta\in [\bt^\mu]$ and $\bt^\gamma\in [\bt^\nu]$.
Since $I(G)$ is a lattice ideal with a positive associated lattice, 
according to \cite[Theorem~4.12 and Remark~4.13]{CTV1}, it suffices to show that when
$\GG(F_\alpha)$ is not a single vertex it is a spanning tree.
\smallskip

\noindent
Fix $\alpha\in \NN^s$, and take $\bt^\beta,\bt^\gamma\in F_\alpha$ such that $\bt^\beta\not \equiv\bt^\gamma$. Then $\bt^\beta-\bt^\gamma$ is a homogeneous binomial in $I(G)$ and, denoting by $d$ its degree, it does not belong to the ideal generated by the elements of $I(G)$ of degree less than $d$.
By Proposition~\ref{thrm: necessary condition for minimal binomials}, there are two cases to consider. In the first case, $\bt^\beta-\bt^\gamma=t_e^2-t_\ell^2$, for some $e,\ell\in E_G$.
Now, if $\bt^{2\xi}\bt_J\in F_\alpha$, then $\bt^{2\xi}\bt_J-t_e^2\in I(G)$ and thus by 
Proposition~\ref{prop: on the fibers} 
the monomial $\bt^{2\xi}\bt_J$ has degree $2$ and $J$ is an $(\emptyset,0)$-join.
We deduce that $J=\emptyset$ and $\bt^{2\xi}\bt_J$ is the square of a variable.
We conclude that
$F_\alpha=\{t_e^2 : e\in E_G\}$. Since, using Proposition~\ref{thrm: necessary condition for minimal binomials}, $t_e^2 \not \equiv t_\ell^2$, for every $e\neq\ell \in E_G$, the vertex set of
$\GG(F_\alpha)$ can be canonically identified with $F_\alpha$ and the edges of 
$\GG(F_\alpha)$ are of the form $\{t_e^2,t_\ell^2\}$, with $e\in E_G\setminus \ell$.
This shows that $\GG(F_\alpha)$ is a spanning tree.

\smallskip

In the second case, $\bt^\beta-\bt^\gamma=\bt_J-\bt_K$, where $J$ and $K$ are two nonequivalent 
minimum cardinality $(T,p)$-joins. Now, if $\bt^{2\xi}\bt_L$ is a monomial in $F_\alpha$, then
$\bt^{2\xi}\bt_L-\bt_J$ is in $I(G)$ and, by Proposition~\ref{prop: on the fibers}, $L$ is a $(T,p)$-join. This implies that $\xi=0$ and hence $L$ is a minimum cardinality $(T,p)$-join.
Therefore $F_\alpha=\{\bt_L : L\in \mJ (T,p)\}$. By Proposition~\ref{thrm: necessary condition for minimal binomials}, given $\bt_L,\bt_M \in F_{\alpha}$ we have 
$$\bt_L\equiv\bt_M \iff L\sim M.$$
We conclude that the vertex set of $\GG(F_\alpha)$ is 
$
\{[\bt_L] : L\in \CS (T,p)\}
$ and the edge set of $\GG(F_\alpha)$ consists of the edges 
$\{[\bt_L],[\bt_{K(T,p)}]\}$, with $L\in \CS (T,p)\setminus K(T,p)$, which again shows that 
$\GG(F_\alpha)$ is a spanning tree.
\end{proof}

\begin{corollary}\label{cor: characterization of minimal generating degrees}
The degrees of the minimal homogeneous generators of $I(G)$ are $2$ and the  cardinalities of minimum cardinality 
$(T,p)$-joins for $(T,p)\in \TG^*$.
\end{corollary}
\begin{proof}
Use Theorem~\ref{thrm: the minimal generating set of IG}.
\end{proof}

\section{Maximal generating degree}\label{sec: Generating degrees} 
Let us denote the maximal generating degree of $I(G)$ by 
$\dd(I(G))$. Before we can address the proofs of the main results of 
this section, we need to show two preliminary results.

\begin{lemma}\label{lem: lemma on (T,p)}
For $(T,p)\in \TG^*$, take $J$ and $K$ to be nonequivalent minimum cardinality $(T,p)$-joins. 
Then $J\dcup K$ is either an even cycle 
or the edge-disjoint union of two odd cycles, with at most one vertex in common.
\end{lemma}
\begin{proof}
Since $J\dcup K$ is a nonempty Eulerian set, by \cite[Theorem 1]{BoMGT}, it is the edge-disjoint union of 
cycles. In particular, it contains a cycle.
Let us consider two cases. 
In the first case, suppose that $J\dcup K$ contains an even cycle.
Denote it by $D$ and let us show that $J\dcup K=D$. 
Since $|D|=|J\cap D|+|K\cap D|$, by Proposition~\ref{prop: char parity joins}, 
we deduce that
$|J|=|J \Delta D|$.
Hence
$J \Delta D$ is a minimum cardinality $(T,p)$-join, different from $J$.
If $J \Delta D\neq K$, 
the sequence $J,J \Delta D, K$ implies that $J$ and $K$ are equivalent. This is a contradiction. Thus
$J \Delta D= K$ and $J\dcup K=D$.
For the second case, let us assume that $J\dcup K$ does not contain any even cycles. 
Then, as $|J\dcup K|$ is even, it contains at least two edge-disjoint odd cycles, 
$C_1,C_2 \subseteq E_G$.
Arguing as in the previous case, we obtain 
that $J\dcup K=C_1 \dcup C_2$.
To see that $C_1$ and $C_2$
have at most one vertex in common, assume otherwise, that they intersect in two distinct 
vertices.
These vertices determine, in each $C_i$, two paths with lengths of different parity. 
Out of these four paths, any two of the same parity form an even cycle contained 
in $J\dcup K$. 
This is a contradiction, so $C_1$ and $C_2$ have at most one vertex in common.
\end{proof}

\begin{definition}\label{def: even-chord}
Let $C\subseteq E_G$ be a nonempty even cardinality Eulerian set.
An edge  
\mbox{$\ell\in E_G\setminus C$} is called an even-chord of $C$ if  
there are even cardinality Eulerian sets $C_1,C_2$ with $C_1\cap C_2=\{\ell\}$ and 
$C=C_1 \Delta C_2$.
\end{definition}

\begin{rmks}\label{rmks: remark on even-chords}
$(i)$ A chord of a cycle $C\subseteq E_G$ is
an edge, of $G$, between two nonadjacent vertices of $C$.
An even-chord of an even cycle is a chord that separates it 
into two even cycles. 
If $G$ is a bipartite graph, any chord of a cycle separates it
into two even cycles and is thus an even-chord.
$(ii)$ An even cardinality Eulerian set with an 
even-chord need not be an even cycle.
Take the graph of Figure~\ref{fig: even cardinality Eulerian sets with even-chords}.
\begin{figure}[ht]
\begin{center}
\begin{tikzpicture}

\coordinate (P1) at (-1,0);
\coordinate (P2) at (0.8-1,0.8);
\coordinate (P3) at (0.8-1,-0.8);
\coordinate (P4) at (-0.8-1,-0.8);
\coordinate (P5) at (-0.8-1,0.8);

\coordinate (P6) at (2+1,0);
\coordinate (P7) at (2+0.30901699,0.95105651);
\coordinate (P8) at (2+-0.80901699,0.58778525);
\coordinate (P9) at (2+-0.80901699,-0.58778525);
\coordinate (P10) at (2+0.30901699,-0.95105651);

\draw (P1)--(P2);
\draw (P1)--(P3);
\draw (P1)--(P4);
\draw (P1)--(P5);
\draw (P2)--(P3);
\draw (P4)--(P5);
\draw (P2)--(P5);

\draw (P6)--(P7);
\draw (P7)--(P8);
\draw (P8)--(P9);
\draw (P9)--(P10);
\draw (P6)--(P10);
\draw (P6)--(P9);

\foreach \i in {1,...,10} {\fill[color=black] (P\i) circle (2pt);}

\foreach \i in {1} {\draw  (P\i) node[anchor=south] {$\scriptstyle \i$};}
\foreach \i in {2} {\draw  (P\i) node[anchor=west] {$\scriptstyle \i$};}
\foreach \i in {3} {\draw  (P\i) node[anchor=north] {$\scriptstyle \i$};}
\foreach \i in {4} {\draw  (P\i) node[anchor=north] {$\scriptstyle \i$};}
\foreach \i in {5} {\draw  (P\i) node[anchor=east] {$\scriptstyle \i$};}

\foreach \i in {6} {\draw  (P\i) node[anchor=south] {$\scriptstyle \i$};}
\foreach \i in {7} {\draw  (P\i) node[anchor=west] {$\scriptstyle \i$};}
\foreach \i in {8} {\draw  (P\i) node[anchor=south] {$\scriptstyle \i$};}
\foreach \i in {9} {\draw  (P\i) node[anchor=north] {$\scriptstyle \i$};}
\foreach \i in {10} {\draw  (P\i) node[anchor=west] {$\scriptstyle \i$};}

\end{tikzpicture}
\end{center}
\caption{Even cardinality Eulerian sets with even-chords.} 
\label{fig: even cardinality Eulerian sets with even-chords}
\end{figure}
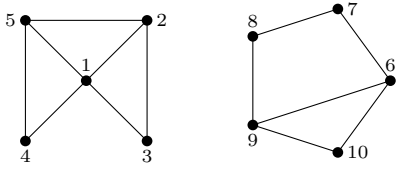 
For this graph, an example of an even cardinality Eulerian subset
with an even-chord is the 
edge-disjoint union of the pentagon on the right with one of the 
triangles on the left, say the subset
$$\{\{6,7\},\{7,8\},\{8,9\},\{9,10\},\{6,10\},\{1,2\},\{1,3\},\{2,3\}\}.$$
The even-chord is the chord of the pentagon, the edge $\{6,9\}$, with $C_1$ 
the square
$\{\{6,7\},\{7,8\},\{8,9\},\{6,9\}\}$ and $C_2$
the vertex-disjoint union of two 
triangles, $$\textstyle\{\{1,2\},\{1,3\},\{2,3\},\{6,9\},\{6,10\},\{9,10\}\}.$$
Another type of example arises as the 
union of two triangles with a vertex in common, say 
$$\{\{1,2\},\{1,3\},\{2,3\},\{1,4\},\{1,5\},\{4,5\}\}.$$
It has the edge $\{2,5\}$ as an even-chord and,
in this case, $C_1$ and $C_2$ are 
the squares 
$\{\{1,3\},\{2,3\},\{2,5\},\{1,5\}\}$
and 
$\{\{1,4\},\{1,2\},\{2,5\},\{4,5\}\}$.
$(iii)$
If $C$ is the even cardinality Eulerian subset given by the vertex-disjoint union of two odd cycles then $C$ has an even-chord if and only if 
one of the odd cycles has a chord. 
To see this, notice that, if $C_1$, $C_2$ and 
$\ell\in E_G\setminus C$ satisfy the conditions of Definition~\ref{def: even-chord} then $C\cup \{\ell\}$ contains $C_1$ and $C_2$. However, if $\ell$ is not 
a chord of either of the odd cycles of $C$, then $\{\ell\}$ does not belong to
any cycle contained in $C\cup \{\ell\}$.
In particular, $C\cup \{\ell\}$ does not contain 
even cardinality Eulerian subsets, other than $C$ itself, contradicting that 
it contains $C_1$ and $C_2$.
Thus $\ell$ must be a chord of one of the odd cycles of $C$. 
Conversely, if $\ell$ is a chord of one of the odd cycles of $C$, it 
separates that cycle into an even cycle and an odd cycle, that intersect in 
$\{\ell\}$.
Let $C_1\subseteq C\cup \{\ell\}$ denote this even cycle and 
let $C_2=C\Delta C_1$, which is the vertex-disjoint union of 
two odd cycles. Then $C_1 \cap C_2=\{\ell\}$ and 
$\ell$ is an even-chord of $C$.
\end{rmks}

\begin{lemma}\label{lem: even-chord}
Let $J$ and $K$ be disjoint minimum cardinality $(T,p)$-joins, for some 
$(T,p)\in \TG$. If $C=J\dcup K$ has an even-chord, then 
$J$ and $K$ are equivalent. 

\end{lemma}
\begin{proof}
Let $C_1,C_2 \subseteq E_G$ be even cardinality Eulerian sets such that  
$C=C_1\Delta C_2$ and $|C_1 \cap C_2|=1$.
If $|J \cap C_i|\leq\frac{|C_i|}{2}-1$,
for both $i=1,2$, we would get that 
$$
\textstyle
\frac{|C_1|}{2}+\frac{|C_2|}{2}-1=\frac{|C|}{2}=|J|=|J \cap C_1|+|J \cap C_2|\leq \frac{|C_1|}{2}+\frac{|C_2|}{2}-2,
$$
which is impossible. Assume then that  
$|J \cap C_1|\geq\frac{|C_1|}{2}$.
By Lemma~\ref{prop: char parity joins}, it follows that
 $|J \cap C_1|=\frac{|C_1|}{2}$ and
$J \Delta C_1$ is a minimum cardinality $(T,p)$-join. 
Let us show that the sequence $J$, $J \Delta C_1$, $K$ gives the equivalence of $J$ and $K$,
concluding the proof. 
We must check that 
$J \setminus C_1 \neq \emptyset$ and
$(C_1 \setminus J)\cap K \neq \emptyset.$
If $J \setminus C_1 = \emptyset$, we get 
$$\textstyle\frac{|C|}{2}=|J|=|J \cap C_1|=\frac{|C_1|}{2},$$
and so  $|C_2|=2$.
This contradicts that nonempty Eulerian sets of even cardinality must have cardinality at least $4$, thus implying that
$J \setminus C_1 \neq \emptyset$.
Now, to see that
$(C_1 \setminus J)\cap K \neq \emptyset$, note that
$$\textstyle|C_1 \setminus J|=|C_1|-|J\cap C_1 |= \frac{|C_1|}{2}\geq 2.$$
So $C_1 \setminus J$ contains the single edge in $C_1\cap C_2$ and at least one other 
edge of $C$, that must belong to $K$.
\end{proof}

\begin{theorem}\label{thrm: complete graphs}
If $G$ is a complete graph, $\dd(I(G))=2$. 
\end{theorem}
\begin{proof}
By Corollary~\ref{cor: characterization of minimal generating degrees}, it 
suffices to show that, if $J$ and $K$ are two nonequivalent minimum cardinality $(T,p)$-joins, then $|J|=|K|\leq 2$.
Assume, with a view to a contradiction, that 
$|J|=|K|\geq 3$.
By Lemma~\ref{lem: lemma on (T,p)}, there are three cases to consider.
If $J\dcup K$ is an even cycle, since $|J\dcup K|\geq 6$ and $G$ is the complete graph, $J\dcup K$ has
an even-chord and,
by Lemma~\ref{lem: even-chord}, $J$ is equivalent to $K$.
If $J\dcup K$ is the edge-disjoint union of two odd cycles, with exactly one vertex in common, say $i\in V_G$, 
then, for any vertices $u,v$ of $J\dcup K$, different from $i$, the edge $\{u,v\}$ is 
an even-chord of $G$.
Again, the result follows, by Lemma~\ref{lem: even-chord}. 
For the last case, $J\dcup K$ is the edge-disjoint union of two odd cycles, $C_1,C_2$, with no vertices in common.
Then $J$ must contain at least half of the edges of one of the two odd cycles, say $C_1$.
In particular, this means that $J$ contains two adjacent edges of $C_1$, 
say $\{u,i\},\{i,v\}\in J$, as in Figure~\ref{fig: Vertex-disjoint triangle and pentagon, $J$ in bold}.
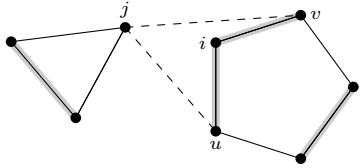
\begin{figure}[ht]
\begin{center}
\begin{tikzpicture}

\coordinate (P1) at (-1,0.79);
\coordinate (P2) at (-1.65,0.79-1.2);
\coordinate (P3) at (-2.5,0.6);
\coordinate (P4) at (1+1,0);
\coordinate (P5) at (1+0.30901699,0.95105651);
\coordinate (P6) at (1+-0.80901699,0.58778525);
\coordinate (P7) at (1+-0.80901699,-0.58778525);
\coordinate (P8) at (1+0.30901699,-0.95105651);

\draw [gray!40!white, line width=3pt] (P2)--(P3);
\draw [gray!40!white, line width=3pt] (P5)--(P6);
\draw [gray!40!white, line width=3pt] (P6)--(P7);
\draw [gray!40!white, line width=3pt] (P4)--(P8);

\draw (P1)--(P2);
\draw (P1)--(P2);
\draw (P1)--(P3);
\draw (P2)--(P3);
\draw (P4)--(P5);
\draw (P5)--(P6);
\draw (P6)--(P7);
\draw (P7)--(P8);
\draw (P4)--(P8);
\draw [dashed] (P1)--(P5);
\draw [dashed] (P1)--(P7);

\draw (P5)--(P6);
\draw (P6)--(P7);
\draw (P4)--(P8);

\foreach \i in {1,...,8} {\fill[color=black] (P\i) circle (2pt);}

\foreach \i in {1} {\draw  (P\i) node[anchor=south] {$\scriptstyle j$};}
\foreach \i in {5} {\draw  (P\i) node[anchor=west] {$\scriptstyle v$};}
\foreach \i in {6} {\draw  (P\i) node[anchor=east] {$\scriptstyle i$};}
\foreach \i in {7} {\draw  (P\i) node[anchor=north] {$\scriptstyle u$};}

\end{tikzpicture}
\end{center}
\caption{Vertex-disjoint triangle and pentagon, $J$ in gray.} 
\label{fig: Vertex-disjoint triangle and pentagon, $J$ in bold}
\end{figure} 
Fix a vertex $j$ of $C_2$ and let $D$ be the square $\{\{u,i\},\{i,v\},\{v,j\},\{j,u\}\}$.
Then $J\Delta D$ is a minimum cardinality $(T,p)$-join, equivalent to $J$, and 
$(J\Delta D) \dcup K$ is the edge-disjoint union of two odd cycles with exactly 
the vertex $j$ in common.
By the previous case, $(J\Delta D) \dcup K$ has an even-chord and 
$J\sim J\Delta D \sim K$.
\end{proof}

Below we adapt part of the proof of
\cite[Theorem 1.2]{OH1999}.

\begin{theorem}\label{thrm: d(IG)=d}
Let $2d$ be the greatest cardinality of an even cycle without an even-chord, 
or an edge-disjoint union of two odd cycles with at most one vertex in common without an even-chord, of $G$.
Then $\dd(I(G))\leq d,$ and equality holds if $G$ is bipartite.
\end{theorem}
\begin{proof}
By Corollary~\ref{cor: characterization of minimal generating degrees}, 
$\dd(I(G))$ is the greatest cardinality of a minimum cardinality $(T,p)$-join, 
among all pairs $(T,p)\in \TG^*$.
Then, by Lemmas~\ref{lem: lemma on (T,p)} and \ref{lem: even-chord}, 
$2\dd(I(G))$ is the cardinality of an even cycle without an even-chord, or 
an edge-disjoint union of two odd cycles with at most one vertex in common without an even-chord, of $G$. 
Thus  \mbox{$\dd(I(G))\leq d.$}
Now, if $G$ is bipartite, let $C$ be a chordless even cycle with $|C|=2d$.
Set $T$ as the set of vertices of $C$ and $p=d+2\ZZ$.
Let $C=J\dcup K$ be the partition of $C$ into two sets, of cardinality $d$, 
taking for $J$ every other edge of $C$.
Then $J$ and $K$ are $(T,p)$-joins.
Take the binomial $\bt_J-\bt_K\in I(G)$.
If $\dd(I(G))< d,$ then $\bt_J-\bt_{K}$ is a $\KK[E_G]$-linear combination of binomials of lower degree. 
It follows that there is an even cycle $D\neq C,$ and a binomial $\bt_A-\bt_{D\setminus A} \in I(G)$ such that 
$A\subseteq J$.
Then, since no edges of $J$ are incident, the same applies to $A$ and 
consequently to $D\setminus A$.
Therefore, as $A\subseteq C$, every vertex of $D$ is 
a vertex of $C$ and every edge in $D\setminus C$ is a chord of $C$, which is a contradiction.
We conclude that $\dd(I(G))= d.$
\end{proof}

Recall that a bipartite graph is called bipartite chordal if every cycle of length greater than $4$ 
has a chord.

\begin{corollary}\label{cor: chordal bipartite graphs}
If $G$ is bipartite, $\dd(I(G))=2$ if and only if 
$G$ is bipartite chordal.
\end{corollary}
\begin{proof}
Since $G$ is bipartite, every chord is an even-chord.
Let $2d$ be the greatest cardinality of an even cycle, of $G$, without a chord.
By Theorem~\ref{thrm: d(IG)=d}, 
$\dd(I(G))=d$ and $d=2$ if and only if $G$ is chordal bipartite. 
\end{proof}

If $G$ is the complete graph on $n\geq6$ vertices, an
edge-disjoint union of two triangles, with no common vertices, does not have an even-chord, as we mentioned in $(iii)$ of Remarks~\ref{rmks: remark on even-chords}.
However, by
Theorem~\ref{thrm: complete graphs}, $\dd(I(G))=2$ and hence the inequality in 
Theorem~\ref{thrm: d(IG)=d} is  not sharp. The following is another example of this.

\begin{example}\label{exa: strict inequality}
Let $G$ be the graph depicted in Figure~\ref{fig: A nonbipartite graph}.
 \begin{figure}[ht]
\begin{center}
\begin{tikzpicture}

\coordinate (P1) at (.5,0.8660254038);
\coordinate (P2) at (1,0);
\coordinate (P3) at (.5,-0.8660254038);
\coordinate (P4) at (-.5,-0.8660254038);
\coordinate (P5) at (-1,0);
\coordinate (P6) at (-.5,0.8660254038);

\draw (P1)--(P2);
\draw (P1)--(P6);
\draw (P2)--(P3);
\draw (P3)--(P4);
\draw (P4)--(P5);
\draw (P5)--(P6);
\draw (P1)--(P3);
\draw (P1)--(P5);
\draw (P3)--(P5);
\draw (P2)--(P4);
\draw (P2)--(P6);
\draw (P4)--(P6);

\foreach \i in {1,...,6} {\fill[color=black] (P\i) circle (2pt);}

\foreach \i in {1} {\draw  (P\i) node[anchor=south] {$\scriptstyle \i$};}
\foreach \i in {2} {\draw  (P\i) node[anchor=west] {$\scriptstyle \i$};}
\foreach \i in {3} {\draw  (P\i) node[anchor=north] {$\scriptstyle \i$};}
\foreach \i in {4} {\draw  (P\i) node[anchor=north] {$\scriptstyle \i$};}
\foreach \i in {5} {\draw  (P\i) node[anchor=east] {$\scriptstyle \i$};}
\foreach \i in {6} {\draw  (P\i) node[anchor=south] {$\scriptstyle \i$};}

\end{tikzpicture}
\end{center}
\caption{A non-bipartite graph.} 
\label{fig: A nonbipartite graph}
\end{figure}
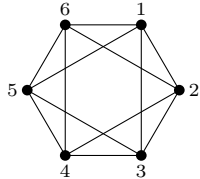
Using Macaulay2, one can check that $\dd(I(G))=2$.
However, the hexagon 
$$\{\{1,2\},\{2,3\},\{3,4\},\{4,5\},\{5,6\},\{1,6\}\}$$ does not have an even-chord, as no edge of $E_G \setminus C$ separates $C$ in two squares.
\end{example}

If $G$ is bipartite, then, by Theorem~\ref{thrm: d(IG)=d} and \cite[Lemmas 3.1, 3.2 and Theorem 1.2]{OH1999},
the maximal generating degree of $I(G)$ and the maximal generating degree of the toric ideal of $G$ coincide. This 
number is half of the length of a largest chordless even cycle.
However, in the non-bipartite case, this does not occur in general. 
Using Macaulay2 one can check that for the graph in Figure~\ref{fig: A nonbipartite graph}, 
the toric ideal of $G$ is generated in degree $2$, but $\dd(I(G))=3$.

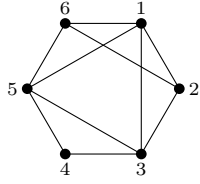
\begin{figure}[h]
\begin{center}
\begin{tikzpicture}

\coordinate (P1) at (.5,0.8660254038);
\coordinate (P2) at (1,0);
\coordinate (P3) at (.5,-0.8660254038);
\coordinate (P4) at (-.5,-0.8660254038);
\coordinate (P5) at (-1,0);
\coordinate (P6) at (-.5,0.8660254038);

\draw (P1)--(P2);
\draw (P1)--(P6);
\draw (P2)--(P3);
\draw (P3)--(P4);
\draw (P4)--(P5);
\draw (P5)--(P6);
\draw (P1)--(P3);
\draw (P1)--(P5);
\draw (P3)--(P5);
\draw (P2)--(P6);

\foreach \i in {1,...,6} {\fill[color=black] (P\i) circle (2pt);}

\foreach \i in {1} {\draw  (P\i) node[anchor=south] {$\scriptstyle \i$};}
\foreach \i in {2} {\draw  (P\i) node[anchor=west] {$\scriptstyle \i$};}
\foreach \i in {3} {\draw  (P\i) node[anchor=north] {$\scriptstyle \i$};}
\foreach \i in {4} {\draw  (P\i) node[anchor=north] {$\scriptstyle \i$};}
\foreach \i in {5} {\draw  (P\i) node[anchor=east] {$\scriptstyle \i$};}
\foreach \i in {6} {\draw  (P\i) node[anchor=south] {$\scriptstyle \i$};}

\end{tikzpicture}
\end{center}
\caption{Another non-bipartite graph.} 
\label{fig: Another nonbipartite graph}
\end{figure}

\end{document}